\documentclass[english]{article}
\usepackage[T1]{fontenc}
\usepackage[latin9]{inputenc}
\usepackage{geometry}
\usepackage{refstyle}
\usepackage{enumitem}
\usepackage{amsmath}
\usepackage{amsthm}
\usepackage{amssymb}
\usepackage{graphicx}
\usepackage{setspace}
\usepackage{nicefrac}
\usepackage{authblk}
\usepackage{units}
\usepackage{esint}

\makeatletter

\AtBeginDocument{\providecommand\secref[1]{\ref{sec:#1}}}
\AtBeginDocument{\providecommand\eqref[1]{\ref{eq:#1}}}
\AtBeginDocument{\providecommand\figref[1]{\ref{fig:#1}}}
\AtBeginDocument{\providecommand\propref[1]{\ref{prop:#1}}}
\AtBeginDocument{\providecommand\claimref[1]{\ref{claim:#1}}}
\AtBeginDocument{\providecommand\thmref[1]{\ref{thm:#1}}}
\AtBeginDocument{\providecommand\remref[1]{\ref{rem:#1}}}
\RS@ifundefined{subref}
  {\def\RSsubtxt{section~}\newref{sub}{name = \RSsubtxt}}
  {}
\RS@ifundefined{thmref}
  {\def\RSthmtxt{theorem~}\newref{thm}{name = \RSthmtxt}}
  {}
\RS@ifundefined{lemref}
  {\def\RSlemtxt{lemma~}\newref{lem}{name = \RSlemtxt}}
  {}

\numberwithin{equation}{section}
\numberwithin{figure}{section}
  \theoremstyle{plain}
  \newtheorem{prop}{\protect\propositionname}
  \theoremstyle{plain}
  \newtheorem{fact}{\protect\factname}
  \theoremstyle{definition}
  \newtheorem{defn}{\protect\definitionname}
\theoremstyle{plain}
\newtheorem{thm}{\protect\theoremname}
\theoremstyle{plain}
\newtheorem{lem}[thm]{\protect\lemmaname}
 \newlist{casenv}{enumerate}{4}
 \setlist[casenv]{leftmargin=*,align=left,widest={iiii}}
 \setlist[casenv,1]{label={{\itshape\ \casename} \arabic*.},ref=\arabic*}
 \setlist[casenv,2]{label={{\itshape\ \casename} \roman*.},ref=\roman*}
 \setlist[casenv,3]{label={{\itshape\ \casename\ \alph*.}},ref=\alph*}
 \setlist[casenv,4]{label={{\itshape\ \casename} \arabic*.},ref=\arabic*}
  \theoremstyle{remark}
  \newtheorem{rem}{\protect\remarkname}
  \theoremstyle{remark}
  \newtheorem{claim}{\protect\claimname}
  \theoremstyle{plain}
  \newtheorem{conjecture}{\protect\conjecturename}

\usepackage{dsfont}
\newref{rem}{refcmd={Remark \ref{#1}}}
\newref{prop}{refcmd={Proposition \ref{#1}}}
\newref{claim}{refcmd={Claim \ref{#1}}}
\newref{cor}{refcmd={Corollary \ref{#1}}}
\newref{sec}{refcmd={Section \ref{#1}}}
\newref{eq}{refcmd={equation \ref{#1}}}
\newref{fig}{refcmd={Figure \ref{#1}}}
\newref{thm}{refcmd={Theorem \ref{#1}}}

\theoremstyle{plain}
\newtheorem{assumption}{Assumption}

\date{}

\makeatother

\usepackage{babel}
  \providecommand{\claimname}{Claim}
  \providecommand{\conjecturename}{Conjecture}
  \providecommand{\definitionname}{Definition}
  \providecommand{\factname}{Fact}
  \providecommand{\propositionname}{Proposition}
  \providecommand{\remarkname}{Remark}
  \providecommand{\lemmaname}{Lemma}
 \providecommand{\casename}{Case}
\providecommand{\theoremname}{Theorem}

\title{Metastable Behavior of Bootstrap Percolation on Galton-Watson Trees}
\author[]{Assaf Shapira\thanks{I acknowledge the support of the ERC Starting Grant 680275 MALIG}}

\begin{document}
\global\long\def\unity{\mathds{1}}

\maketitle
\begin{abstract}
Continuing the study of \cite{bollobas2014bootstrapongwcriticality}
on the critical probability of the bootstrap percolation on Galton-Watson
trees, we analyze the metastable states near criticality. We find
that, depending on the exact choice of the offspring distribution,
it is possible to have several distinct metastable states, with varying
scaling of their duration while approaching criticality.
\end{abstract}

\section{Introduction}

Bootstrap percolation is a deterministic dynamics in discrete time first introduced in \cite{chalupa1979bootstrapbethelattice} in order to model disordered magnetic systems, and broadly studied since in many different contexts. Fix a graph
$G$ and a parameter $r\in\mathbb{N}$. Each vertex of the graph can
be in one of two states -- infected or healthy, which are initially
distributed independently with probabilities $p$ and $q=1-p$. At
each time step we update these states, such that the infected vertices
remain infected, and a healthy vertex becomes infected if it has at
least $r$ infected neighbors. One may also consider more general
infection conditions, such as the oriented bootstrap percolation --
when the graph $G$ is oriented, and we require at least $r$ edges
to point at infected vertices. 

Bootstrap percolation on various deterministic graphs has been the
subject of extensive research. On the grid $\left[n\right]^{d}$,
the probability that all vertices are eventually infected, as a function
of $p$ (or equivalently $q$), has been profoundly studied in \cite{aizenman1988metastabilityz2,holroyd2003sharp,balogh2012sharpzd}.
For $\left(d+1\right)$-regular infinite trees, with $2\le r \le d$, it is shown in \cite{balogh2006bootstrapregulartrees}
that a phase transition occurs. Defining $q_{c}$ to be the supremum
over all $q$ such that starting with probability $q$ to be healthy
all vertices end up being infected with probability $1$, an explicit
expression for $q_{c}$ is found, and it is furthermore proven that
\textbf{$q_{c}$} lies in the open interval $\left(0,1\right)$. In
addition, it is determined, depending on $d$ and on $r$, when the
transition is continuous and when it is discontinuous. In \cite{biskup2009metastableregulartrees}
the details of the metastability properties are studied, describing
the time evolution of the probability that the root stays healthy
near criticality.

Random environments have also been of interest in this field, e.g.,
the bootstrap percolation on a polluted grid \cite{gravner1997polluted,gravner2017polluted},
the random graph $G_{n,p}$ \cite{janson2012Gnp}, the random regular
graph \cite{balogh2007randomregular,janson2009percolationexplosion},
and the Galton-Watson tree \cite{bollobas2014bootstrapongwcriticality}.

In this paper, we will analyze the metastability of the bootstrap
percolation on a directed Galton-Watson tree, i.e., the time behavior
near criticality of the probability that the root is infected. In
\secref{Prevalence} we present an interpretation of this probability
as the almost sure prevalence -- the limiting ratio of infected vertices.
In \secref{Critical-Behavior} we will study the zoology of the metastabilities
for different offspring distributions, showing that this model introduces
a vast variety of possible behaviors. Finally, in \secref{othertransitions}
we comment on other phase transitions that may occur.

\section{Model and Notations}

Fix an infection threshold $r\ge2$, and consider a Galton-Watson
tree $G$ whose offspring distribution is supported on $r,r+1,\dots$
That is, defining $\xi_{k}$ to be the probability that a vertex has
$k$ children, we require $\xi_{k}=0$ for $k<r$.

In the beginning, we decide for each vertex of $G$ whether it is
infected or healthy, independently with probabilities $p$ and $q=1-p$
respectively. Then, at each time step $t$, a healthy vertex will
get infected if it has at least $r$ infected children. Let us denote
by $\phi_{t}^{G}$ the (random) probability that the root is healthy at time
$t$, so in particular $\phi_{0}^{G}=q$. Note also that $\phi_{t}^{G}$
is decreasing in $t$. The expected value of $\phi_{t}^{G}$ over
all graphs $G$, generated with offspring distribution $\xi$, will
be denoted $\phi_{t}^{\xi}$.

One particular case, that has been studied in \cite{balogh2006bootstrapregulartrees,biskup2009metastableregulartrees,fontesschonmann2008bootstraptreesecondtransition,chalupa1979bootstrapbethelattice},
is the case of a rooted $\left(d+1\right)$-regular tree, i.e., $\xi_{k}=\unity_{k=d}$.
Here, one can find $\phi_{t}^{d}$ recursively using the relation
\begin{eqnarray}
\phi_{t+1}^{d} & = & h_{d}\left(\phi_{t}^{d}\right);\label{eq:recursionregular}\\
h_{d}\left(x\right) & = & q\mathbb{P}\left[\mbox{Bin}\left(d,1-x\right)\le r-1\right].\label{eq:defhregular}
\end{eqnarray}
For the GW tree, such a recursion still holds for the expected value
$\phi_{t}^{\xi}$:
\begin{eqnarray}
\phi_{t+1}^{\xi} & = & h_{\xi}\left(\phi_{t}^{\xi}\right);\label{eq:recursion}\\
h_{\xi}\left(x\right) & = & \sum_{k=r}^{\infty}\xi_k h_{k}(x).\label{eq:defh}
\end{eqnarray}

\section{Results}

\subsection{\label{sec:Prevalence}Prevalence and $\phi_{t}$}

The relation in \eqref{recursion} allows us to find the expected
value of $\phi_{t}^{G}$, but for a specific realization of $G$,
$\phi_{t}^{G}$ may differ from that value. For example, fixing $t$,
there is a nonzero probability that a finite neighborhood of the root
will have many vertices of high degree, which will result in a smaller
$\phi_{t}^{G}$. However, we will see that $\phi_{t}^{\xi}$ describes
almost surely another observable -- the prevalence, i.e., the limiting
fraction of infected vertices.

First, denote by $B(R)$ the ball of radius $R$ around the root.
We can then define the $R$-prevalence at time $t$ as
\[
\rho_{R}\left(t\right)=\frac{\left|\left\{ \mbox{infected vertices in }B(R)\mbox{ at time }t\right\} \right|}{\left|B(R)\right|}.
\]

It is natural to expect $\rho_{R}\left(t\right)$ to be close to $1-\phi_{t}^{\xi}$,
and this is indeed the case, as shown in the following proposition:
\begin{prop}
\label{prop:prevalence}Fix $t$. Then $\lim_{R\rightarrow\infty}\rho_{R}\left(t\right)=1-\phi_{t}^{\xi}$
almost surely (in both the graph and the initial state measures).
\end{prop}

\subsection{\label{sec:Critical-Behavior}Critical Behavior}

Following \cite{balogh2006bootstrapregulartrees,bollobas2014bootstrapongwcriticality},
we define the critical probability
\begin{equation}
q_{c}=\sup_{\left[0,1\right]}\left\{ q:\,\phi_{\infty}^{\xi}=0\right\} .\label{eq:qcdef}
\end{equation}

In order to analyze this criticality, define
\begin{eqnarray}
g_{k}\left(x\right) & = & \frac{h_{k}\left(x\right)}{qx},\label{eq:defofgk}\\
g_{\xi}\left(x\right) & = & \frac{h_{\xi}\left(x\right)}{qx}.\label{eq:defofgxi}
\end{eqnarray}

In \cite{bollobas2014bootstrapongwcriticality}, the following fact is shown:
\begin{fact} \label{fact:gandcriticallity}
Fix $\xi$. Then:
\begin{enumerate}
\item For a given $q$, $\phi_{\infty}^{\xi}$ is the maximal solution in
$\left[0,1\right]$ of $g_{\xi}\left(x\right)=\frac{1}{q}$, and $0$
if no such solution exists.
\item $q_{c}=\frac{1}{\max_{\left[0,1\right]}g_{\xi}\left(x\right)}$.
\end{enumerate}
\end{fact}
We will consider here the behavior near criticality, at $q$ slightly
smaller than $q_{c}$.
\begin{defn}
For $0<x<1$ and some positive $\delta$, the $\delta$-entrance time of $x$ is
\[
\tau^-_{x,\delta}(q) = \min\{t:\,\phi^\xi_t<x+\delta\},
\]
and the $\delta$-exit time is defined as
\[
\tau^+_{x,\delta}(q) = \min\{t:\,\phi^\xi_t<x-\delta\}.
\]

\end{defn}
\begin{defn}
Fix $\delta>0$. We say that the critical point
is $\delta$-$\left(\nu_{1},\dots,\nu_{n}\right)$-metastable at $x_{1}>\dots>x_{n}>0$
if, for $q\nearrow q_{c}$, the following hold:
\begin{enumerate}
\item $\tau_{x_1,\delta}^{-}=O\left(1\right)$.
\item $\frac{\log\left(\tau_{x_i,\delta}^{+}-\tau_{x_i,\delta}^{-}\right)}{\log\left(q_{c}-q\right)}\xrightarrow{q\nearrow q_{c}}-1+\frac{1}{2\nu_{i}}$
for $i=1,\dots,n$.
\item $\tau_{x_{i+1},\delta}^{-}-\tau_{x_i,\delta}^{+}=O\left(1\right)$ for $i=1,\dots,n$ and $x_{n+1}=0$.
\end{enumerate}
We say that the critical point is $\left(\nu_{1},\dots,\nu_{n}\right)$-metastable at $x_{1}>\dots>x_{n}$ if it is $\delta$-$\left(\nu_{1},\dots,\nu_{n}\right)$-metastable at $x_{1}>\dots>x_{n}$ for small enough $\delta$. See \figref{multiplemetastabilities}.
\begin{figure}
\includegraphics[width=1\linewidth]{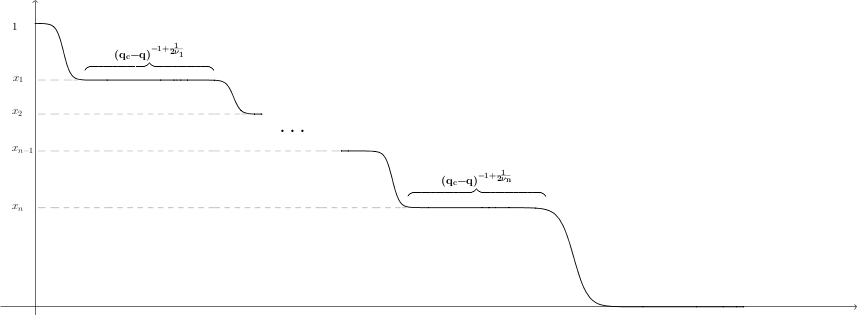}

\caption{\label{fig:multiplemetastabilities}A schematic picture of $\phi_{t}^{\xi}$
as a function of $t$ for a $\left(\nu_{1},\dots,\nu_{n}\right)$-metastable
criticality at $x_{1}>\dots>x_{n}$. }
\end{figure}

\end{defn}
The following theorem gives a full classification of the metastability
properties:
\begin{thm}
\label{thm:zeologyofmetastabilities}Fix $\xi$. Then the metastable
behavior is determined by one of the following cases:
\begin{casenv}
\item $g_{\xi}$ attains its maximum at $1$. In this case the critical probability
is $1$.
\item $g_{\xi}$ has a unique maximum at $0$. In this case the phase transition
is continuous. At the critical point
\begin{equation}
\frac{\log(\phi_{t}^{\xi})}{\log t} \xrightarrow{t \rightarrow \infty} -\frac{1}{\nu},\label{eq:criticalbehaviorcontinuous}
\end{equation}
where $\nu$ is determined by the asymptotic expansion $g_{\xi}\left(x\right)=\frac{1}{q_{c}}-Cx^{\nu}+o\left(x^{\nu}\right)$.

\item The maximum of $g_{\xi}$ is attained at the points $x_{1},\dots,x_{n}$ for $1>x_{1}>\dots>x_{n}>0$, and possibly also at $0$.
In this case the phase transition is discontinuous. For $i=1,\dots,n$ we
may write around $x_{i}$
\begin{equation}
g_{\xi}\left(x\right)=\frac{1}{q_{c}}-C_{i}\left(x-x_{i}\right)^{2\nu_{i}}+o\left(\left(x-x_{i}\right)^{2\nu_{i}}\right),\label{eq:asymptoticsaroundmax}
\end{equation}
with some $C_i>0$.

Then the critical point is $\left(\nu_{1},\dots,\nu_{n}\right)$-metastable
at $x_{1}>\dots>x_{n}$.

\end{casenv}
\end{thm}
\begin{rem}
In the first case, where the critical probability is $1$, it is not
clear whether or not an asymptotic expansion exists, since $g_{\xi}$
is not guaranteed to be analytic. When it does exist, one can recover
a result similar to Case 3.
\end{rem}
Finally, we show the main result -- that the different metastability
behaviors described above are possible:
\begin{thm}
\label{thm:everythingispossible}~\end{thm}
\begin{enumerate}
\item Let $\nu\in\mathbb{N}$. Then there exists $\xi$ such that the phase
transition is continuous, and satisfies \eqref{criticalbehaviorcontinuous}
at criticality.
\item Let $\left(\nu_{1},\dots,\nu_{n}\right)\in\mathbb{N}^{n}$. Then there
exist $\xi$ and $x_{1}>\dots>x_{n}$ such that the critical point is $\left(\nu_{1},\dots,\nu_{n}\right)$-metastable
at $x_{1}>\dots>x_{n}$.
\end{enumerate}

\section{Proofs}
\begin{proof}
[Proof of \propref{prevalence}]The idea of the proof is to notice
that the main contribution to the prevalence comes from the
sites close to the boundary, and then use their independence. Thus,
we fix a width $w$, and consider
\[
\rho_{R,w}\left(t\right)=\frac{\left|\left\{ \mbox{infected vertices in }B\left(R\right)\setminus B\left(R-w\right)\mbox{ at time }t\right\} \right|}{\left|B\left(R\right)\setminus B\left(R-w\right)\right|}.
\]

First, we claim that $\rho_{R}\left(t\right)$ is approximated by
$\rho_{R,w}\left(t\right)$ for large $w$. More accurately, we have
$\left|B\left(R-w\right)\right|\le2^{-w}\left|B(R)\right|$, which also implies that
the number of infected vertices in $B\left(R\right)\setminus B\left(R-w\right)$
is the same as the number of infected vertices in $B\left(R\right)$,
up to a correction of order $2^{-w}|B(R)|$. Then
\begin{equation}
\rho_{R}\left(t\right)=\rho_{R,w}\left(t\right)+O\left(2^{-w}\right).\label{eq:prevalenceonanulusclosetototal}
\end{equation}

We would now like to bound the distance between $\rho_{R,w}\left(t\right)$
and $1-\phi_{t}^{\xi}$. Let $\varepsilon>0$, and, by \eqref{prevalenceonanulusclosetototal},
take $w$ big enough such that $\left|\rho_{R}\left(t\right)-\rho_{R,w}\left(t\right)\right|<\frac{\varepsilon}{2}$ uniformly in $R$.
Note that $\rho_{R,w}\left(t\right)$ is a weighted average of the
$w$ random variables $\rho_{R,1}\left(t\right),\rho_{R-1,1}\left(t\right),\dots,\rho_{R-w+1,1}\left(t\right)$,
and consider one of these variables, $\rho_{r,1}\left(t\right)$.
This variable is the average of the random variables $\unity_{v\,\mbox{\tiny{is infected}}}$
for all vertices $v$ of distance $r$ from the root, and since these
are independent Bernoulli random variables with mean $1-\phi_{t}^{\xi}$,
and since there are at least $2^{R-w+1}$ such variables, we can use
a large deviation estimate, yielding

\[
\mathbb{P}\left[\left|\rho_{r,1}(t)-\left(1-\phi_{t}^{\xi}\right)\right|>\frac{\varepsilon}{2}\right]\le e^{-c\,2^{R-w+1}}
\]
for a positive $c$ that only depends on $\varepsilon$ and on $\phi_{t}^{\xi}$.
Since for $1-\phi_{t}^{\xi}$ to be far from $\rho_{R,w}\left(t\right)$
it must be far from at least one of the variables $\rho_{R,1}\left(t\right),\rho_{R-1,1}\left(t\right),\dots,\rho_{R-w+1,1}\left(t\right)$,
we have

\begin{equation}
\mathbb{P}\left[\left|\rho_{R,w}\left(t\right)-\left(1-\phi_{t}^{\xi}\right)\right|>\frac{\varepsilon}{2}\right]\le we^{-c\,2^{R-w+1}}.\label{eq:prevalenceonanulusclosetoaverage}
\end{equation}

Hence, $\rho_{R}(t)$
is $\varepsilon$-close to $1-\phi_{t}^{\xi}$ with probability larger
than $1-we^{-c\,2^{R-w+1}}$, which concludes the proof by the Borel-Cantelli
lemma.
\end{proof}
Before proving Theorems \ref{thm:zeologyofmetastabilities} and \ref{thm:everythingispossible},
we will need a couple of small results.

\begin{claim}
\label{claim:gkispolynomial}$g_{k}$ is a polynomial of degree $k-1$,
whose lowest degree monomial is of degree $k-r$.\end{claim}
\begin{proof}
By equations \ref{eq:defofgk} and \ref{eq:defhregular}
\begin{eqnarray*}
g_{k}\left(x\right) & = & \frac{\mathbb{P}\left[\mbox{Bin}\left(k,1-x\right)\le r-1\right]}{x}\\
 & = & \sum_{i=0}^{r-1}{k \choose i}\left(1-x\right)^{i}x^{k-i-1};
\end{eqnarray*}
therefore all monomials are of degree between $k-r$
and $k-1$. The coefficient of $x^{k-r}$ is ${k \choose r-1}\neq0$,
and the coefficient of $x^{k-1}$ is $\sum_{i=0}^{r-1}{k \choose i}\left(-1\right)^{i}$,
which is also nonzero since $0<r-1<k$. This concludes the proof.\end{proof}
\begin{claim}
\label{claim:basisofpolynomials}$g_{r}\left(x\right),\dots,g_{m}\left(x\right),x^{m-r+1},\dots,x^{m-1}$
is a basis of the linear space of polynomials of degree smaller or
equal to $m-1$.\end{claim}
\begin{proof}
Denote $v_{1}\left(x\right)=g_{r}\left(x\right),\dots,v_{m-r+1}\left(x\right)=g_{m}\left(x\right),v_{m-r+2}\left(x\right)=x^{m-r+1},v_{m}\left(x\right)=x^{m-1}$.
By \claimref{gkispolynomial}, all $v$'s are of degree smaller or
equal to $m-1$. Moreover, the matrix whose $\left(i,j\right)$ entry
is the coefficient of $x^{j}$ in the polynomial $v_{i}$ is upper
triangular, with nonzero diagonal. This shows that $\left\{ v_{i}\right\} _{i=1}^{m}$
is indeed a basis. 
\end{proof}

We will also use the following result from \cite{bollobas2014bootstrapongwcriticality}:
\begin{claim}
[Claim 3.9 of \cite{bollobas2014bootstrapongwcriticality}]\label{claim:constantg}For
$\xi_{k}=\frac{r-1}{k\left(k-1\right)}$, $g_{\xi}\left(x\right)=1$.
\end{claim}
We are now ready to prove Theorems \ref{thm:zeologyofmetastabilities}
and \ref{thm:everythingispossible}.
\begin{proof}
[Proof of \thmref{zeologyofmetastabilities}]First, we note that
$g_{k}\left(1\right)=1$ for all $k$, so in particular the series
$\sum_{k=r}^{\infty}\xi_{k}g_{k}\left(x\right)$ converges at $x=1$.
By \claimref{gkispolynomial}, the monomials of degree up to $n$ of the partial sum $\sum_{k=r}^{N}\xi_{k}g_{k}\left(x\right)$ are fixed once $N>n+r$.
From these two facts we
conclude that $g_\xi(x)$ is analytic in $\left(-1,1\right)$
and continuous at $1$. Thus, cases 1, 2 and 3 exhaust all possibilities.

The result will then follow from general arguments of dynamical systems
near a bifurcation point. Since the exact calculations are a bit tedious, we only give here a short sketch of the argument, referring to the appendix for the complete proof.

For case 2, the expression
\[\phi_{t+1} = \phi_{t} - Cq_c \phi_t^{\nu+1} + o(\phi_t^{\nu+1})
\]
could be estimated by comparing to the differential equation 
\[
\frac{\text{d}\phi}{\text{d}t} = - Cq_c \phi_t^{\nu+1}.
\]
This equation could be solved explicitly, yielding the asymptotics of \eqref{criticalbehaviorcontinuous}.

For case 3, the approximate differential equation will be
\[
\frac{\text{d}\phi}{\text{d}t} = - \frac{x_i}{q_c}(q_c-q) -C_i q_c x_i (\phi - x_i)^{2\nu_i}.
\]
The solution of this equation has a plateau around $x_i$, whose length diverges as $(q_c-q)^{-1+\frac{1}{2\nu_i}}$.
\end{proof}
~
\begin{proof}
[Proof of \thmref{everythingispossible}] For the first part, it
will be enough to show that there exist an offspring distribution
$\xi$ and a polynomial $Q(x)=b_{0}+\dots+b_{r-2}x^{r-2}$ such that
\begin{enumerate}
\item $g_{\xi}\left(x\right)=\text{Const}-x^{\nu}Q(x)$.
\item $Q(x)>0$ for all $x\in\left[0,1\right]$.
\end{enumerate}

This $\xi$, according to \thmref{zeologyofmetastabilities} and the
fact that $b_{0}>0$, will indeed satisfy \eqref{criticalbehaviorcontinuous}.
Rather than $\xi$, it will be easier to find a sequence $\left\{ \chi_{k}\right\} _{k=r}^{\infty}$
with a finite sum together with a polynomial $P(x)=a_{0}+\dots+a_{r-2}x^{r-2}$,
such that
\begin{enumerate}
\item $g_{\chi}\left(x\right)=\sum_{k}\chi_{k}g_{k}\left(x\right)=1-x^{\nu}P\left(x\right)$.
\item $\chi_{k}\ge0$.
\item $P\left(x\right)>0$ for all $x\in\left[0,1\right]$.
\end{enumerate}

Taking $\xi=\frac{1}{\sum\chi_{k}}\chi_{k}$ will then conclude the
proof. Let

\begin{equation}
\chi_{k}=\begin{cases}
\frac{r-1}{k\left(k-1\right)} & r\le k\le\nu+r-1\\
0 & k\ge\nu+r
\end{cases}.\label{eq:chidef}
\end{equation}

Using \claimref{constantg}, we may write

\[
g_{\chi}\left(x\right)=1-\sum_{k=\nu+r}^{\infty}\frac{r-1}{k\left(k-1\right)}g_{k}\left(x\right).
\]

By \claimref{gkispolynomial} $g_{\chi}$ is a polynomial of degree
$\nu+r-2$, therefore $\sum_{k=\nu+r}^{\infty}\frac{r-1}{k\left(k-1\right)}g_{k}\left(x\right)$
equals a polynomial of degree $\nu+r-2$. Using again \claimref{gkispolynomial},
we can define the polynomial

\[
P\left(x\right)=\sum_{k=\nu+r}^{\infty}\frac{r-1}{k\left(k-1\right)}\frac{g_{k}\left(x\right)}{x^{\nu}}.
\]

It is left to show that $P\left(x\right)>0$ for all $x\in\left[0,1\right]$.
By equations \ref{eq:defofgk} and \ref{eq:defhregular}, $P(x)$
is non-negative and could only vanish at $x=0$. But by \claimref{gkispolynomial},
$P(0)=\frac{r-1}{(\nu+r)(\nu+r-1)}\left(\frac{g_{\nu+r}(x)}{x^{\nu}}\right)_{x=0}\neq0$.
This concludes the first part.
\begin{rem}
\label{rem:uniquesolforP}Note that, by \claimref{basisofpolynomials},
we can define the projection Pr from the space of polynomials of degree
at most $r+\nu-2$ to its subspace spanned by $x^{\nu},\dots,x^{\nu+r-2}$
with kernel spanned by $g_{r}(x),\dots,g_{\nu+r-1}(x)$. Define also
$M_{0}$ to be the map from the space of polynomials of degree at
most $r-2$ to the space of polynomials of degree at most $r+\nu-2$
given by the multiplication by \textbf{$x^{\nu}$}. Then the first
of the conditions above can be written as
\[
\mbox{Pr}M_{0}\,P=\mbox{Pr }1.
\]

Since $\mbox{Pr}\circ M_{0}$ is bijective, this equation has a unique
solution; and what we have shown in the proof is that this solution
satisfies the necessary positivity conditions.
\end{rem}

We will now prove the second part of the theorem. In analogy with
the first one, we will find $\overline{\xi}$, $\overline{Q}(x)=\overline{b}_{0}+\dots+\overline{b}_{r-2}x^{r-2}$
and $x_{1}>\dots>x_{n}$ such that:
\begin{enumerate}
\item $g_{\overline{\xi}}\left(x\right)=\overline{\text{Const}}-\left(x-x_{1}\right)^{2\nu_{1}}\dots\,\left(x-x_{n}\right)^{2\nu_{n}}\overline{Q}\left(x\right)$.
\item $\overline{Q}(x)>0$ for all $x\in\left[0,1\right]$.
\end{enumerate}

Similarly to the previous part, we will look for $\left\{ \overline{\chi}_{k}\right\} _{k=r}^{\nu+r-1}$
and $\overline{P}(x)=\overline{a}_{0}+\dots+\overline{a}_{r-2}x^{r-2}$ satisfying:
\begin{enumerate}
\item $g_{\overline{\chi}}\left(x\right)=\sum_{k}\overline{\chi}_{k}g_{k}\left(x\right)=1-\left(x-x_{1}\right)^{2\nu_{1}}\dots\,\left(x-x_{n}\right)^{2\nu_{n}}\overline{P}\left(x\right)$.
\item $\overline{\chi}_{k}>0$.
\item $\overline{P}\left(x\right)>0$ for all $x\in\left[0,1\right]$.
\end{enumerate}

Note that choosing $\nu=2\nu_{1}+\dots+2\nu_{n}$, $\chi_{k}$ (defined
in \eqref{chidef}) is strictly positive for $r\le k\le\nu+r-1$. Since $P$ was required to be strictly positive, we may hope that also after adding a small perturbation $\left(x_{1},\dots,x_{n}\right)$
around $0$ there still exists a positive solution $\overline{P}$. More precisely,
let us denote by $M_{x_{1},\dots,x_{n}}$ the multiplication by $(x-x_{1})^{2\nu_{1}}\dots\,(x-x_{n})^{2\nu_{n}}$, acting on the polynomials of degree at most $r-2$.
In particular, for $x_{1},\dots,x_{n}=0$ this is $M_{0}$ defined
in \remref{uniquesolforP}. Then, we want to show that the solution
of
\[
\mbox{Pr}M_{x_{1},\dots,x_{n}}\,\overline{P}=\mbox{Pr }1
\]
satisfies the positivity conditions 2 and 3.
By continuity of the determinant, when $(x_1,\dots,x_n)$ is in a small neighborhood of $0$ the operator $\mbox{Pr}M_{x_{1},\dots,x_{n}}$ is invertible. Moreover, in an even smaller neighborhood of $0$ the polynomial $\left(\mbox{Pr}M_{x_1,\dots,x_n}\right)^{-1}\mbox{Pr}1$ will satisfy the positivity condition 3 -- matrix inversion is continuous, and the set of polynomials satisfying this condition is open and contains $\left(\mbox{Pr}M_0\right)^{-1}\mbox{Pr}1$ by the first part of the proof.
Finally, since coordinate projections of $1 - (x-x_{1})^{2\nu_{1}}\dots\,(x-x_{n})^{2\nu_{n}}\left(\mbox{Pr}M_{x_1,\dots,x_n}\right)^{-1}\mbox{Pr}1$ with respect to the basis defined in \claimref{basisofpolynomials} are continuous in $(x_1,\dots,x_n)$, and since for $(x_1,\dots,x_n)=0$ condition 2 is satisfied, by taking $(x_1,\dots,x_n)$ in a further smaller neighborhood of $0$ we are guaranteed to find a polynomial $\overline{P}$ satisfying the required conditions.
\end{proof}

\section{\label{sec:othertransitions}Remarks on Two Other Phase Transitions}

\subsection{More Discontinuities of $\phi_{t}$}

Consider, for example, $r=2$ and $\xi_{k}=\frac{3}{5}\unity_{k=2}+\frac{2}{5}\unity_{k=5}$.
The function $g_{\xi}\left(x\right)$ is maximal at $g_{\xi}\left(0\right)=\frac{6}{5}$,
then it has a local minimum, followed by a local maximum (see \figref{secondphasetransition}).
\begin{figure}
\includegraphics[width=1\linewidth]{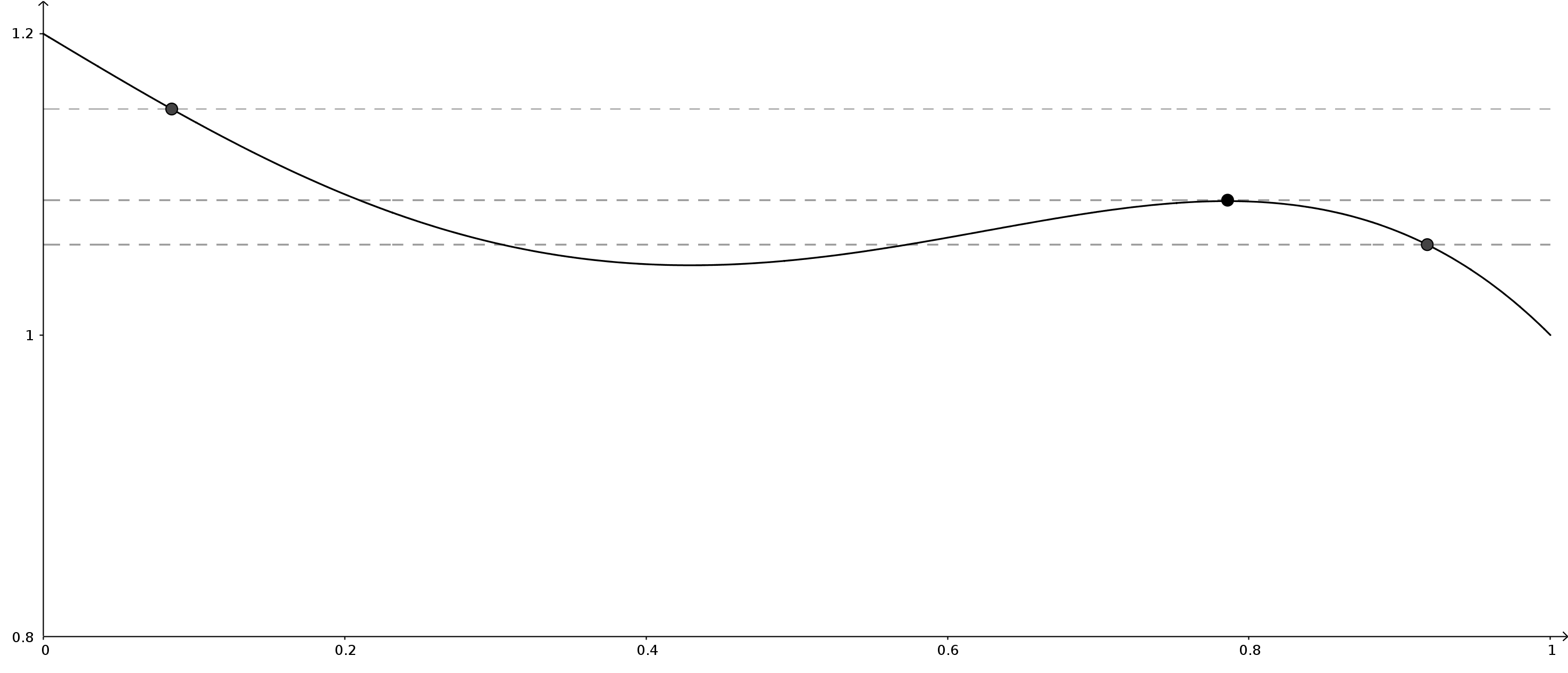}

\caption{\label{fig:secondphasetransition}$g_{\xi}$ for $r=2$ and $\xi_{k}=\frac{3}{5}\protect\unity_{k=2}+\frac{2}{5}\protect\unity_{k=5}$.
We show three lines $\frac{1}{q}$ for three parameters $q$, intersecting
$g_{\xi}$ at $\phi_{\infty}^{\xi}$. One sees here the discontinuity
when $\frac{1}{q}$ equals the value of $g_{\xi}$ at the local maximum.}

\end{figure}
 In this case, recalling Fact \ref{fact:gandcriticallity}, $\phi_{t}^{\xi}$ will have a discontinuity at this
local maximum, that is, a second phase transition occurs. We may then
expect that one can find $\xi$ giving rise to as many (decreasing)
local maxima of $g_{\xi}$ as we wish:
\begin{conjecture}
\label{conj:otherphasetransitions}Let $\nu_{1}^{\left(1\right)},\dots,\nu_{n_{1}}^{\left(1\right)},\nu_{1}^{\left(2\right)},\dots,\nu_{n_{2}}^{\left(2\right)},\dots,\nu_{n_{m}}^{\left(m\right)}$.
Then there exists $g_{\xi}$, $\left\{ q_{i}\right\} _{i=1}^{m},\left\{ x_{j}^{\left(i\right)}\right\} _{1\le i\le m,\,1\le j\le n_{i}}$
such that $q_{i}$ is a critical point which is $\left(\nu_{1}^{\left(i\right)},\dots,\nu_{n_{i}}^{\left(i\right)}\right)$-metastable
at $x_{1}^{\left(i\right)},\dots,x_{n_{i}}^{\left(i\right)}$.
\end{conjecture}

\subsection{Percolation of Infection}

Another possible phase transition, studied in \cite{fontesschonmann2008bootstraptreesecondtransition}
for the case of regular trees, is when infinite infected clusters
start to appear, but the prevalence is still smaller than $1$. Following
the proof of Proposition 3.9 in  \cite{fontesschonmann2008bootstraptreesecondtransition},
one sees that it applies also for the bootstrap percolation on GW
trees, showing that the critical probability $q_{c}^{\left(\infty\right)}$
above which infinite clusters no longer appear is strictly bigger
than $q_{c}$ defined in \eqref{qcdef}, unless $\xi_{k}=\unity_{r}$.

\section{More Questions}

The problem of bootstrap percolation in disordered systems raises many
questions. Related to the work presented here, one may be interested
in the metastable regime for other systems, such as $G_{n,p}$ or
the random regular graph. Another natural problem is the analysis
of the bootstrap percolation on the random graph with a given degree
sequence, that has a GW local structure, with analogy to the regular
tree structure of the random regular graph.

\section*{Acknowledgments}

I would like to thank Cristina Toninelli for the introduction of the
subject and for useful discussion, and to Lucas Benigni for his help
throughout the writing of this paper.

\appendix
\section*{Appendix}
\renewcommand{\theequation}{A\arabic{equation}}
\renewcommand{\thethm}{A\arabic{thm}}
\renewcommand{\theassumption}{A\arabic{assumption}}
\renewcommand{\thefact}{A\arabic{fact}}
\renewcommand{\thelem}{A\arabic{lem}}
\renewcommand{\thedefn}{A\arabic{defn}}

\setcounter{thm}{0}
\setcounter{fact}{0}
\setcounter{defn}{0}

This paper concerns with the analysis of a phase transition originating
in the appearance of a new fixed point for a certain recurrence relation,
i.e., a bifurcation. In this appendix, we will try to understand in
a more general context the time scaling in systems of that type. Let
us then consider a sequence of reals $\left\{ x_{n}\right\} _{n=0}^{\infty}$
, defined by the value $x_{0}$ and a recursion formula for $n>0$:
\begin{equation}\label{eq:appendixrecursion}
x_{n}=f(x_{n-1}).
\end{equation}

We will also fix now some positive $\delta<1$, that will be used
throughout this appendix as the window around the new fixed point
in which we are interested.

First, we will study the time scaling at the bifurcation point, when
the new fixed point is first created. In this case, we may expect
$f$ to be tangent to the identity function at the fixed point, so
we will start our discussion with the following assumptions:

\begin{assumption}\label{asmp:nearfixedpointexponent}

$f$ has a fixed point $y_{0}$, such that for $y\in\left(y_{0},y_{0}+\delta\right)$:

\[
y-\underline{c}\left(y-y_{0}\right)^{\alpha}\le f(y)\le y-\overline{c}\left(y-y_{0}\right)^{\alpha},
\]
for some $\alpha>1$, $0<\overline{c}\le\underline{c}<\delta^{-\left(\alpha-1\right)}$.

\end{assumption}

\begin{assumption}\label{asmp:startingclosetofixedpoint}

$x_{0}\in\left(y_{0},y_{0}+\delta\right)$.

\end{assumption}

We first mention the following fact:
\begin{fact}
\label{fact:decreasingandbounded}The sequence is decreasing and bounded
from below by $y_{0}$.\end{fact}
\begin{proof}
By Assumption \ref{asmp:nearfixedpointexponent}, $x_{n+1}<x_{n}$
whenever $x_{n}\in\left(y_{0},y_{0}+\delta\right)$. Moreover,
\begin{eqnarray*}
x_{n+1}-y_{0} & \ge & x_{n}-y_{0}-\underline{c}\left(x_{n}-y_{0}\right)^{\alpha}\\
 & = & \left(x_{n}-y_{0}\right)\left(1-\underline{c}\left(x_{n}-y_{0}\right)^{\alpha-1}\right)\\
 & \ge & \left(x_{n}-y_{0}\right)\left(1-\underline{c}\delta^{\alpha-1}\right)\\
 & > & 0.
\end{eqnarray*}

Therefore, since $x_{0}\in\left(y_{0},y_{0}+\delta\right)$ by assumption
\ref{asmp:startingclosetofixedpoint}, the entire sequence is in the
interval $\left(y_{0},y_{0}+\delta\right)$, and it is decreasing.
\end{proof}
The following theorem will describe the asymptotic of the sequence:
\begin{thm}
\label{thm:timescalingatbifurcationpoint}
Let $\{x_n\}_{n=0}^\infty$ be the sequence defined in \eqref{appendixrecursion}, satisfying Assumptions \ref{asmp:nearfixedpointexponent} and \ref{asmp:startingclosetofixedpoint}. Then
\[
y_{0}+\underline{a}\left(n+n_{0}\right)^{-\frac{1}{\alpha-1}}\le x_{n}\le y_{0}+\overline{a}n^{-\frac{1}{\alpha-1}},
\]
where $\underline{a}=\left[\left(\alpha-1\right)\left(1-\delta \right)^{-\alpha}\underline{c}\right]^{-\frac{1}{\alpha-1}}$,
$\overline{a}=\left[(\alpha-1)\overline{c}\right]^{-\frac{1}{\alpha-1}}$,
and $n_{0}=\frac{(x_0-y_0)^{1-\alpha}}{\left(\alpha-1\right)\left(1-\delta \right)^{-\alpha}\underline{c}}$ are
all positive constants.\end{thm}
\begin{proof}
Let us first define a sequence $t_{n}=\left(x_{n}-y_{0}\right)^{1-\alpha}$,
and note that $t_{n}$ is positive for all $n$. Then using Fact \ref{fact:decreasingandbounded}
and Assumption \ref{asmp:nearfixedpointexponent}, fixing \textbf{$\underline{c}^{\prime}=\left(\alpha-1\right)\left(1-\delta \right)^{-\alpha}\underline{c}$}
and $\overline{c}^{\prime}=\left(\alpha-1\right)\overline{c}$, we can estimate:

\begin{center}
\begin{tabular}{cc}
$t_{n}=\left(f\left(x_{n-1}\right)-y_{0}\right)^{1-\alpha}$\hphantom{aaaaaaaaaaaaaa} & $t_{n}=\left(f\left(x_{n-1}\right)-y_{0}\right)^{1-\alpha}$\hphantom{aaaaaaaaaaaaaa}\tabularnewline
$\le\left(x_{n-1}-\underline{c}\left(x_{n-1}-y_{0}\right)^{\alpha}-y_{0}\right)^{1-\alpha}$ & $\ge\left(x_{n-1}-\overline{c}\left(x_{n-1}-y_{0}\right)^{\alpha}-y_{0}\right)^{1-\alpha}$\tabularnewline
$=\left(t_{n-1}^{\frac{1}{1-\alpha}}-\underline{c}t_{n-1}^{\frac{\alpha}{1-\alpha}}\right)^{1-\alpha}$\hphantom{aaaaaaaaaaa} & $=\left(t_{n-1}^{\frac{1}{1-\alpha}}-\overline{c}t_{n-1}^{\frac{\alpha}{1-\alpha}}\right)$\hphantom{aaaaaaaaaaaaaa}
\tabularnewline
$=t_{n-1}\left(1-\underline{c}t_{n-1}^{-1}\right)^{1-\alpha}$\hphantom{aaaaaaaaaaa} & $=t_{n-1}\left(1-\overline{c}t_{n-1}^{-1}\right)^{1-\alpha}$\hphantom{aaaaaaaaaaa}
\tabularnewline
$\le t_{n-1}\left(1+\underline{c}^{\prime}t_{n-1}^{-1}\right)$\hphantom{aaaaaaaaaaaaa} & $\ge t_{n-1}\left(1+\overline{c}^{\prime}t_{n-1}^{-1}\right)$\hphantom{aaaaaaaaaaaaa}
\tabularnewline
$=t_{n-1}+\underline{c}^{\prime}$;\hphantom{aaaaaaaaaaaaaaaaaAa} & $=t_{n-1}+\overline{c}^{\prime}$.\hphantom{aaaaaaaaaaaaaaaaaAa}\tabularnewline
\end{tabular}
\par\end{center}
We have used here the fact that, for any $0<z<\delta<1$, we can approximate
$\left(1-z\right)^{1-\alpha}$ using its derivatives at $0$ and at $\delta$:
\[
-(1-\alpha) \le \frac{\left(1-z\right)^{1-\alpha}-1}{z} \le -(1-\alpha) (1-\delta)^{-\alpha}.
\]
We then also use $\underline{c}t_{n-1}^{-1}=\left(x_{n}-y_{0}\right)^{\alpha-1}<\delta^{\alpha-1}<\delta$.

Finally,

\begin{center}
\begin{tabular}{cc}
$x_{n}=y_{0}+t_{n}^{-\frac{1}{\alpha-1}}$
\hphantom{aaaaaaaaaaaaaaaaaaaaa} & $x_{n}\ge y_{0}+\left(\left(x_{0}-y_{0}\right)^{1-\alpha}+\underline{c}^{\prime}n\right)^{-\frac{1}{\alpha-1}}$
\tabularnewline
$\le y_{0}+\left(\left(x_{0}-y_{0}\right)^{1-\alpha}+\overline{c}^{\prime}n\right)^{-\frac{1}{\alpha-1}}$ & $= y_{0}+\left(\underline{c}^{\prime}\left(n+\frac{(x_0-y_0)^{1-\alpha}}{\underline{c}^{\prime}}\right)\right)^{-\frac{1}{\alpha-1}}$\hphantom{a}
\tabularnewline
$\le y_{0}+\overline{a}n^{-\frac{1}{\alpha-1}};$
\hphantom{aaaaaaaaaaaaaaa} & $=y_{0}+\underline{a}\left(n+n_{0}\right)^{-\frac{1}{\alpha-1}}.$\hphantom{aaaaaaaaaa}
\tabularnewline
\end{tabular}
\par\end{center}

\end{proof}
Next, we will be interested in the behavior near the bifurcation point,
just before the new fixed point appears. For this purpose we will
consider a family $\left\{ x_{n}^{\varepsilon}\right\} _{n=0}^{\infty}$
of sequences, each defined by the value $x_{0}^{\varepsilon}$ and
a recursion formula for $n>0$:
\begin{equation}\label{eq:appendixrecursion2}
x^\varepsilon_{n}=f_{\varepsilon}\left(x^\varepsilon_{n-1}\right),
\end{equation}
and assume:

\begin{assumption}\label{asmp:nearfixedpointexponent2}

There is a point $y_{0}$ such that for $\left|y-y_{0}\right|<\delta$ and $\varepsilon<\varepsilon_{0}$

\[
y-\underline{c}\left(y-y_{0}\right)^{2\alpha}-\varepsilon\le f_{\varepsilon}\left(y\right)\le y-\overline{c}\left(y-y_{0}\right)^{2\alpha}-\varepsilon,
\]
for an integer $\alpha>1$ and positive constants $\overline{c}$ and $\underline{c}$.

\end{assumption}

\begin{assumption}\label{asmp:startingclosetofixedpoint2}

$0<x_{0}-y_{0}<\delta$.

\end{assumption}

In order to study the asymptotic behavior of $x^\varepsilon_{n}$ for small values
of $\varepsilon$, we will need the following definition:
\begin{defn}\label{def:exittime}
The exit time $N_{\delta}(\varepsilon)$ is the minimal $n$ such
that $x^\varepsilon_{n}<y_{0}-\delta$.

Replacing Fact \ref{fact:decreasingandbounded} will be the following:\end{defn}
\begin{fact}
\label{fact:decreasingtoinf}For all $\varepsilon<\varepsilon_{0}$,
$N_{\delta}(\varepsilon)$ is finite, and for $n<N_{\delta}(\varepsilon)$
the sequence $x^\varepsilon_{n}$ is decreasing.\end{fact}
\begin{proof}
By Assumption \ref{asmp:nearfixedpointexponent2}, for $n<N_{\delta}(\varepsilon)$, if
$x^\varepsilon_{n}<y_{0}+\delta$ then $x^\varepsilon_{n+1}<x^\varepsilon_{n}<y_{0}+\delta$. Hence, the sequence remains in the interval $(y_0-\delta,y_0+\delta)$ an long as $n<N_{\delta}(\varepsilon)$.
Since in this interval the sequence is decreasing, the result follows by Assumption \ref{asmp:startingclosetofixedpoint2}.
\end{proof}
For our analysis, we will compare this sequence to the solution of
the following differential equations, that will approximate $x^\varepsilon_{n}-y_{0}$:

\begin{center}
\begin{tabular}{cc}
$\frac{\mbox{d}\underline{\ensuremath{\zeta}}}{\mbox{d}s}=-\underline{c}\,\underline{\ensuremath{\zeta}}^{2\alpha}-\varepsilon,$\hphantom{.} & $\frac{\mbox{d}\overline{\zeta}}{\mbox{d}s}=-\overline{c}\,\overline{\zeta}^{2\alpha}-\varepsilon,$\hphantom{.}\tabularnewline
$\underline{\ensuremath{\zeta}}\left(0\right)=z^\varepsilon_{0}=x^\varepsilon_{0}-y_{0};$ & $\overline{\zeta}\left(0\right)=z^\varepsilon_{0}=x^\varepsilon_{0}-y_{0}.$\tabularnewline
\end{tabular}
\par\end{center}

The solution $\overline{\zeta}$ is strictly decreasing, and in particular
one can define its inverse $\overline{t}:\,\left[-\infty,z^\varepsilon_{0}\right]\rightarrow\left[0,\infty\right]$,
and $\overline{\tau}_{n}=\overline{t}\left(x^\varepsilon_{n}-y_{0}\right)$.
$\underline{t}$ and $\underline{\tau}_n$ will be defined in the same manner.
Note that these all depend on $\varepsilon$, even though this dependence is omitted from the notation.
The next lemma will show that the continuous crossing times $\overline{\tau}_n$ and $\underline{\tau}_n$ are close to the discrete one, namely $n$.
\begin{lem}
\label{lem:comparetodiffeq}For all $n\le N_{\delta}(\varepsilon)$,
\[
\left(1-\kappa_{\overline{c},\delta,\varepsilon}\right)n\le\overline{\tau}_{n}\le\underline{\tau}_{n}\le\left(1+\kappa_{\underline{c},\delta,\varepsilon}\right)n,
\]
where for all $c>0$, $\kappa_{c,\delta,\varepsilon_0} = \max(C_4 \varepsilon^{2\alpha-1},2\alpha\delta ^{2\alpha-1})$. $C_4$ is a positive constant depending on $\delta,c$ and $\varepsilon_0$ given explicitly in the proof, and bounded when $\delta$ and $\varepsilon_0$ are not too big. For example, if $\varepsilon_0<1$ and $c\delta^{2\alpha-1}<\frac{1}{2}$, $C_4 < (3+ 4^\alpha c)^{4\alpha}$.
\end{lem}
\begin{proof}
Let $z_{n}=x_{n}-y_{0}$. Then
\begin{eqnarray*}
\underline{\tau}_{n} & = & \underline{t}\left(f_\varepsilon \left(x_{n-1}\right)-y_{0}\right)\\
 & \le & \underline{t}\left(z_{n-1}-\underline{c}z_{n-1}^{2\alpha}-\varepsilon\right)\\
 & = & \int\limits _{z_{0}}^{z_{n-1}-\underline{c}z_{n-1}^{2\alpha}-\varepsilon}\frac{\mbox{d}z}{-\underline{c}z^{2\alpha}-\varepsilon}\\
 & = & \underline{t}\left(z_{n-1}\right)-\int\limits _{z_{n-1}}^{z_{n-1}-\underline{c}z_{n-1}^{2\alpha}-\varepsilon}\frac{\mbox{d}z}{\underline{c}z_{n-1}^{2\alpha}+\varepsilon}-\int\limits _{z_{n-1}}^{z_{n-1}-cz_{n-1}^{2\alpha}-\varepsilon}\left(\frac{\mbox{d}z}{\underline{c}z^{2\alpha}+\varepsilon}-\frac{\mbox{d}z}{\underline{c}z_{n-1}^{2\alpha}+\varepsilon}\right)\\
 & = & \underline{\tau}_{n-1}+1-\int\limits _{z_{n-1}}^{z_{n-1}-\underline{c}z_{n-1}^{2\alpha}-\varepsilon}\left(\frac{\mbox{d}z}{\underline{c}z^{2\alpha}+\varepsilon}-\frac{\mbox{d}z}{\underline{c}z_{n-1}^{2\alpha}+\varepsilon}\right).
\end{eqnarray*}

In order to study the error term, we will use the following estimation:

\begin{claim}
Fix $w_{0}\in(-\delta,\delta)$, and $c>0$. Let
\[
I=\int\limits _{w_{0}-cw_{0}^{2\alpha}-\varepsilon}^{w_{0}}\left(\frac{1}{cw^{2\alpha}+\varepsilon}-\frac{1}{cw_{0}^{2\alpha}+\varepsilon}\right) \mbox{d}w.
\]
Then
\[
\left|I\right|\le \kappa_{c,\delta,\varepsilon_0}.
\]
\end{claim}
\begin{proof}
We will first consider the case in which the integration interval passes through
$0$, that is $0<w_{0}<cw_{0}^{2\alpha}+\varepsilon$. In this case,

\begin{eqnarray*}
w_{0} & \le & w_{0}\left(1-cw_{0}^{2\alpha-1}\right)\left(1-c\delta^{2\alpha-1}\right)^{-1}\le C_{1}\varepsilon,\\
cw_{0}^{2\alpha}+\varepsilon & \le & \left[1+C_{2}\varepsilon^{2\alpha-1}\right]\varepsilon,
\end{eqnarray*}

for $C_{1}=\left(1-c\delta^{2\alpha-1}\right)^{-1}$ and $C_{2}=c\left(1-c\delta^{2\alpha-1}\right)^{-2\alpha}$.

We may then bound the nominator of the integrand for all $w\in\left[w_{0}-cw_{0}^{2\alpha}-\varepsilon,w_{0}\right]$
by

\[
\left|cw_{0}^{2\alpha}+\varepsilon-cw^{2\alpha}-\varepsilon\right|\le cw_{0}^{2\alpha}+cw^{2\alpha}\le C_{3}\varepsilon^{2\alpha},
\]

where $C_{3}=\left(1+C_{2}\varepsilon_{0}^{2\alpha-1}\right)^{2\alpha}+C_{1}^{2\alpha}$.

For the denominator,

\[
\left(cw^{2\alpha}+\varepsilon\right)\left(cw_{0}^{2\alpha}+\varepsilon\right)\ge\varepsilon^{2}.
\]

Putting everything together

\begin{eqnarray*}
\left|I\right| & \le & \int\limits _{w_{0}-cw_{0}^{2\alpha}-\varepsilon}^{w_{0}}\left|\frac{cw_{0}^{2\alpha}+\varepsilon-cw^{2\alpha}-\varepsilon}{\left(cw^{2\alpha}+\varepsilon\right)\left(cw_{0}^{2\alpha}+\varepsilon\right)}\right|\\
 & \le & \left(cw_{0}^{2\alpha}+\varepsilon\right)C_{3}\varepsilon^{2\alpha-2}\\
 & \le & C_{4}\varepsilon^{2\alpha-1},
\end{eqnarray*}

for $C_{4}=\left[1+C_{2}\varepsilon_{0}^{2\alpha-1}\right]C_{3}$.

Next, we consider the case where the integral is over a positive interval,
i.e., $w_{0}\ge cw_{0}^{2\alpha}+\varepsilon$. We can bound the integrand
using convexity -- for all $w\in\left(w_{0}-cw_{0}^{2\alpha}-\varepsilon,w_{0}\right)$

\[
\frac{\frac{1}{cw^{2\alpha}+\varepsilon}-\frac{1}{cw_{0}^{2\alpha}+\varepsilon}}{w-w_{0}}\ge-\frac{2\alpha c w^{2\alpha-1}}{\left(cw^{2\alpha}+\varepsilon\right)^{2}}.
\]

This implies that

\begin{eqnarray*}
|I| & \le & \left(cw_{0}^{2\alpha}+\varepsilon\right)\frac{2\alpha c w^{2\alpha-1}}{\left(cw^{2\alpha}+\varepsilon\right)^{2}}\left(w_{0}-w\right)\\
 & \le & 2\alpha c w^{2\alpha-1}\le 2\alpha c \delta^{2\alpha-1}.
\end{eqnarray*}

We are left with the case $w_{0}\le-cw_{0}^{2\alpha}-\varepsilon$,
which could be analyzed using the exact same argument as the previous
one to obtain the result.
\end{proof}

Using the claim we can continue with our estimation, obtaining
\[
\underline{\ensuremath{\tau}}_{n}\le\underline{\ensuremath{\tau}}_{n-1}+1+\kappa_{\underline{c},\delta,\varepsilon_0},
\]
and proving the upper bound. The lower bound could be found using the
exact same calculation replacing $\underline{c}$ by $\overline{c}$. The result follows since $\overline{c}\le\underline{c}$, and thus
$\overline{\tau}_{n}\le\underline{\tau}_{n}$ by monotonicity of the
integral.
\end{proof}
We are now ready to formulate the final result:
\begin{thm}
Fix a family of sequences (indexed by $\varepsilon$) defined in \eqref{appendixrecursion2} satisfying Assumptions \ref{asmp:nearfixedpointexponent2} and \ref{asmp:startingclosetofixedpoint2}, and consider their exit times $N_{\delta}(\varepsilon)$ (see Definition \ref{def:exittime}).
\label{thm:plateaulength}Let $\overline{I}=\int\limits _{-\infty}^{\infty}\frac{\mbox{d}u}{\overline{c}u^{2\alpha}+1}$,
$\underline{I}=\int\limits _{-\infty}^{\infty}\frac{\mbox{d}u}{\underline{c}u^{2\alpha}+1}$,
and $\kappa_{\delta,0}=\max(\kappa_{\overline{c},\delta,0},\kappa_{\underline{c},\delta,0})$, where $\kappa_{\overline{c},\delta,0}$ and $\kappa_{\underline{c},\delta,0}$ are the positive constants given in Lemma \ref{lem:comparetodiffeq}. Assume further that $\kappa_{\delta,0}<1$. Then

\[
0<\frac{\frac{1}{2}\underline{I}}{\left(1+\kappa_{\delta,0}\right)}\le\liminf_{\varepsilon\rightarrow0}\frac{N_{\delta}(\varepsilon)}{\varepsilon^{-1+\frac{1}{2\alpha}}}\le\limsup_{\varepsilon\rightarrow0}\frac{N_{\delta}(\varepsilon)}{\varepsilon^{-1+\frac{1}{2\alpha}}}\le\frac{\overline{I}}{\left(1-\kappa_{\delta,0}\right)}<\infty.
\]
The factor $\frac{1}{2}$ in front of $\underline{I}$ could be removed when $\varepsilon^{-\frac{1}{2\alpha}}(x^\varepsilon_0-y_0)^\varepsilon \rightarrow \infty$ as $\varepsilon \rightarrow 0$ (e.g., when $x^\varepsilon_0-y_0$ is bounded away from $0$ uniformly in $\varepsilon$).
\end{thm}
\begin{proof}
This theorem is a direct consequence of the fact that $\zeta$ shows
an $\varepsilon^{-1+\frac{1}{2\alpha}}$ time scaling behavior. First, note that
\[
\overline{\tau}_{N_{\delta}(\varepsilon)-1}\le\overline{t}(-\delta)\le\underline{t}(-\delta)\le\underline{\ensuremath{\tau}}_{N_{\delta}(\varepsilon)}.
\]
 We will then be interested in finding $\overline{t}(-\delta),\underline{t}(-\delta)$:
\begin{eqnarray*}
t(-\delta) & = & \int\limits _{z^\varepsilon_{0}}^{-\delta}\frac{\mbox{d}z}{-cz^{2\alpha}-\varepsilon}\\
 & = & \varepsilon^{-1+\frac{1}{2\alpha}}\int\limits _{z_0^\varepsilon}^{-\delta}\frac{\varepsilon^{-\frac{1}{2\alpha}}\mbox{d}z}{-c\left(z\varepsilon^{-\frac{1}{2\alpha}}\right)^{2\alpha}-1}\\
 & = & \varepsilon^{-1+\frac{1}{2\alpha}}\int\limits _{-\varepsilon^{-\frac{1}{2\alpha}}z^\varepsilon_0}^{\varepsilon^{-\frac{1}{2\alpha}}\delta}\frac{\mbox{d}u}{cu^{2\alpha}+1},
\end{eqnarray*}
where for $\overline{t}$ one should take $c=\overline{c}$, and $c=\underline{c}$
for $\underline{t}$.

All that is left is to use Lemma \ref{lem:comparetodiffeq}, finding

\[
\frac{\underline{t}(-\delta)}{\left(1+\kappa_{\delta,\varepsilon}\right)}\le N_{\delta}(\varepsilon)\le1+\frac{\overline{t}(-\delta)}{\left(1-\kappa_{\delta,\varepsilon}\right)},
\]
which, since the integrals defining $\overline{I}$ and $\underline{I}$
converge, concludes the proof.\end{proof}
\begin{rem}
When $f_{\varepsilon}$ satisfies not only Assumption \ref{asmp:nearfixedpointexponent2},
but also
\[
f_{\varepsilon}\left(y\right)=y-c\left(y-y_{0}\right)^{2\alpha}-\varepsilon+o\left(\left(y-y_{0}\right)^{2\alpha}\right)+o\left(\varepsilon\right),
\]
we can consider $\delta_{\varepsilon}$ that goes to $0$ with $\varepsilon$,
e.g. $\frac{1}{\left|\log\varepsilon\right|}$, so that $\kappa_{\delta,0}$ will converge to $0$ as well. In this case, we may
choose $\overline{c}_{\delta}$ and $\underline{c}_{\delta}$ that converge
to $c$, and thus Theorem \ref{thm:plateaulength} will give the limit
of $\frac{N_{\delta}(\varepsilon)}{\varepsilon^{-1+\frac{1}{2\alpha}}}$,
rather than just bounds on its limsup and liminf. Such a direct application
of the theorem, however, forces us to choose an initial condition
$x_{0}^{\varepsilon}$ that converges to $y_{0}$ as $\varepsilon$
goes to $0$. To overcome this issue, we can use the estimation above
with a fixed $\delta$ until $x_{n}$ reaches $\delta_{\varepsilon}$,
which happens at $n$ of order $\int\limits _{z_{0}}^{\delta_{\varepsilon}}\frac{\mbox{d}z}{-cz^{2\alpha}-\varepsilon}\ll\varepsilon^{-1+\frac{1}{2\alpha}}$.
Then restart the dynamics using the estimation with $\delta_{\varepsilon}$
until reaching $-\delta_{\varepsilon}$, which takes an order $\varepsilon^{-1+\frac{1}{2\alpha}}$
of steps, and then using again the estimation for our fixed $\delta$
show that the number of steps required to reach $-\delta$ is much
smaller than $\varepsilon^{-1+\frac{1}{2\alpha}}$. This would yield

\[
\lim_{\varepsilon\rightarrow0}\frac{N_{\delta}(\varepsilon)}{\varepsilon^{-1+\frac{1}{2\alpha}}}=\int\limits _{-\infty}^{\infty}\frac{\mbox{d}u}{cu^{2\alpha}+1}.
\]
\end{rem}

\bibliographystyle{plain}
\bibliography{bootstrap_percalation_on_gw}

\end{document}